\documentclass[a4paper]{article}

\usepackage{graphicx,color}
\usepackage{a4wide}
\usepackage{amssymb}
\usepackage{amsfonts}
\usepackage{amsmath}
\usepackage{amssymb}
\usepackage{amsthm}

\newtheorem{theorem}{Theorem}[section]

\newtheorem{lemma}[theorem]{Lemma}

\newtheorem{proposition}[theorem]{Proposition}

\theoremstyle{definition}

\newtheorem{remark}[theorem]{Remark}

\usepackage{mathrsfs}
\renewcommand{\S}{\mathscr{S}}

\newcommand{\R}{\mathbb{R}}
\newcommand{\supp}{\mathop{\textup{supp}}}
\newcommand{\singsupp}{\mathop{\textup{singsupp}}}

\newcommand{\WF}{\mathop{\textup{WF}}}
\newcommand{\Char}{\mathop{\textup{Char}}}
\newcommand{\sgn}{\mathop{\textup{sgn}}}

\newcommand{\F}{\mathcal{F}}

\begin{document}

\title{Solvability and microlocal analysis of the\\ fractional Eringen wave equation}
\author{
 G\"{u}nther H\"{o}rmann\thanks{Faculty of Mathematics, University of Vienna, Oskar-Morgenstern-Platz 1, 1090 Wien, Austria,
 guenther.hoermann@univie.ac.at}, 
 Ljubica Oparnica\thanks{Faculty of Education in Sombor, University of Novi Sad, Podgori\v{c}ka 4, 25000 Sombor, Serbia,
 ljubica.oparnica@gmail.com},
 Du\v{s}an Zorica\thanks{Mathematical Institute, Serbian Academy of Arts and Sciences, Kneza Mihaila 36, 11000 Belgrade, Serbia, dusan\textunderscore zorica@mi.sanu.ac.rs
 and Department of Physics, Faculty of Sciences, University of Novi Sad, Trg D. Obradovi\'{c}a 4, 21000 Novi Sad, Serbia}}

\maketitle

\begin{abstract}
\noindent We discuss unique existence and microlocal regularity properties of Sobolev space solutions to the fractional Eringen wave equation, initially given in the form of a system of equations in which the classical non-local Eringen constitutive equation is generalized by employing space-fractional derivatives. Numerical examples illustrate the shape of solutions in dependence of the order of the space-fractional derivative. 
\end{abstract}

\section{Introduction} 
In this work we analyse the generalized Cauchy problem for the fractional Eringen wave equation  given by: 
\begin{equation}\label{EringenFWE}
\partial_{t}^{2}u(x,t)-L^\alpha_x\partial_x^2 
u(x,t)=u_0(x)\otimes \delta'(t) + v_0(x)\otimes \delta(t),\;\;(x,t)\in \mathbb{R}^2,
\end{equation}
with $\supp(u)$ being contained in forward time $t>0$. The operator $L^\alpha_x$ is of convolution type with respect to the space variable and acts on a temperate
distribution $w$, such that $\xi \mapsto \frac{\widehat{w}}{\sqrt{1 + a_\alpha |\xi|^\alpha}}$ is temperate, in the following way
\begin{equation}
 L^\alpha_x w(x,t)=\mathcal{F}_{\xi \rightarrow x}^{-1} 
(\frac{1}{\sqrt{1+a_\alpha\left\vert \xi \right\vert^\alpha}})\ast_x w(x,t), \quad \alpha \in \,]1,3[, 
\label{L-alpha}
\end{equation}
with constant $a_\alpha$ given by
\begin{equation}\label{aalpha}
 a_\alpha=-\cos(\frac{\alpha\pi}{2}).
\end{equation}

The fractional Eringen wave equation is obtained from the system of three equations which give relations between  displacement,  stress  and  strain, respectively denoted by $u$, $\sigma $ and $\varepsilon $:
equation of motion of a (one-dimensional) deformable body (\ref{em0}), fractional Eringen
constitutive equation (\ref{FracECE}) and strain (\ref{sm0}). The system of equations reads
\begin{eqnarray}
{\partial_x}\sigma (x,t) &=& \rho\,\partial_t^{2}u(x,t),  \label{em0} \\
\sigma (x,t)-\ell^{\alpha} \, \mathrm{D}^{\alpha} \sigma (x,t)&=& E\,\varepsilon(x,t),
\label{FracECE} \\
\varepsilon (x,t) &=&{\partial_x}u\left( x,t\right),\label{sm0}
\end{eqnarray}
where $\ell$ denotes the parameter of non-locality, $E$ denotes the modulus of 
elasticity and $\rho$ denotes the density of material. 
Further,  $\mathrm{D}^{\alpha }$ denotes the fractional
differential operator defined as:
\begin{equation}\label{FD}
 \mathcal{F}(\mathrm{D}^{\alpha}w)=|\xi|^{\alpha}\cos(\frac{\alpha\pi}{2}) \hat{w},\quad \alpha \in \left(1,3\right). 
 \end{equation}
Note that the fractional derivative  (\ref{FD}) can be expressed through the two types of symmterized fractional derivative
\begin{eqnarray}
\mathrm{D}^{\alpha}w &=& \frac{1}{2\Gamma \left( 2-\alpha \right) }\frac{1}{{{\left\vert
x\right\vert }^{\alpha-1 }}} \ast_{x}\partial_x^2 w, \quad \alpha\in [1,2),\\
\mathrm{D}^{\alpha}w &=& \frac{1}{2\Gamma \left( 3-\alpha \right) }\frac{\sgn x}{{{\left\vert
x\right\vert }^{\alpha-2 }}} \ast_{x}\partial_x^3 w, \quad \alpha\in [2,3),
\end{eqnarray}
taken so that in both cases if $\alpha\rightarrow2$, then $\mathrm{D}^{\alpha}\rightarrow\partial_x^2 $, as expained in \cite{CZAS}.

The constitutive equation (\ref{FracECE}),  introduced in \cite{CZAS}, is the generalization of the stress gradient Eringen constitutive equation
\begin{equation}\label{ECE}
\sigma (x,t)-\ell^2 \partial_x^2 \sigma (x,t)= E\,\varepsilon(x,t),
\end{equation} 
proposed in \cite{eringen1}. For $\ell=0$ in (\ref{FracECE}) and \eqref{ECE} one obtains Hooke's law.

A wave equation that uses stress and strain gradient variant of the Eringen constitutive equation (\ref{ECE}) is considered in \cite{Challamel} for harmonic wave propagation in order to obtain a dispersion equation, used for finding the optimal values of the model parameters by comparison with the Born-K\'{a}rm\'{a}n model of lattice dynamics. By the same method, the optimal value of the order of the fractional differentiation $\alpha$, appearing in \eqref{FracECE}, is obtained in \cite{CZAS} as well. 

A microlocal analysis for the non-local wave equation, with the non-locality expressed through the Riesz fractional derivative, is performed in \cite{HOZ16}. Existence and stability of traveling wave solution to a non-local and non-linear version of the wave equation is studied in \cite{ErbayErbayErkip}, where the non-locality and non-linearity is expressed through pseudodifferential operators. A non-locality of Eringen's type is used in \cite{TongYuHuShiXu} in order to study wave propagation in fluid saturated porous media. Properties like $L^2$-stability and monotone wavefronts of traveling waves and pulses propagation in the case of non-local reaction-diffusion equation are widely studied, see for example \cite{LangStannat,TrofimchukPintoTrofimchuk,Volpert}.

An extensive review of application of non-local stress and strain gradient constitutive equations in vibrations and wave propagation in nano-scale beams is given in \cite{EltaherKhaterEmam}. In the wave propagation problems, including axial and torsional traveling and reflecting waves, the solution is assumed in harmonic form and the dispersion equation is analyzed. 

\section{Solvability and Sobolev regularity properties}\label{Solv}
In this section we establish existence and uniqueness of solutions to the Eringen fractional wave equation (\ref{EringenFWE}), thereby obtaining basic regularity properties in terms of Sobolev scales.

We start by introducing dimensionless quantities in the system (\ref{em0}-\ref{sm0}),
\begin{equation*}
\bar{x}=\frac{x}{\ell},\quad
\bar{t}=\frac{t}{\ell}\sqrt{\frac{E}{\rho}},\quad
\bar{u}=\frac{u}{\ell},\quad
\bar{\sigma}=\frac{\sigma }{E},
\end{equation*}
but for notational convenience drop the bar immediately again, and obtain 
\begin{eqnarray}
{\partial_x}\sigma (x,t) &=& \partial_t^{2}u(x,t),  \label{em} \\
\sigma (x,t)- \mathrm{D}^{\alpha} \sigma (x,t)&=& \varepsilon(x,t),
\label{FracECE1} \\
\varepsilon (x,t) &=&{\partial_x}u\left( x,t\right).\label{sm}
\end{eqnarray}
Applying a Fourier transform with respect to the spatial variable to system (\ref{em} - \ref{sm}) -- assuming that the solutions  are temperate distributions -- and by appropriate substitutions we arrive at 
\begin{equation}
\partial_t^2 \hat{u} + 
   \frac{\xi^2}{1 + a_\alpha |\xi|^\alpha}\, \hat{u}=0. \label{FTequ}
\end{equation}
The fractional Eringen wave equation 
\begin{equation}\label{EFWE}
\partial_{t}^{2}u(x,t)-L^\alpha_x\partial_x^2 
u(x,t)=0,
\end{equation}
where $L^\alpha_x$ is as in (\ref{L-alpha}), is defined via inverse (spatial) Fourier transform in \eqref{FTequ}. 
Suppose we had a classical solution $u_{\text{cl}}$ of (\ref{EFWE}), which is $C^2$ for $t>0$ and of class $C^{1}$ for $t\geq 0$, with initial data 
$$
 u_{\text{cl}} \!\mid_{t=0}\, = u_{0} \in C^{1}(\R), \quad 
 {\partial_t}u_{\text{cl}} \!\mid_{t=0}\, = v_{0}\in C(\R),
$$
and we put $u_{\text{cl}}(x,t) = 0$ for $t<0$.  Then we may define the distribution
\begin{equation*}
u(x,t)=u_{\text{cl}}(x,t) H(t) ,\;\;x,t\in\R,
\end{equation*}
where $H$ denotes the Heaviside function, so that $u$ has support in $t \geq 0$ and satisfies the fractional Eringen wave equation in the form (\ref{EringenFWE}). If $u_0$ and $v_0$ are also temperate this amounts to replacing \eqref{FTequ} by 
\begin{equation}
 \partial_t^2 \hat{u} + 
   \frac{\xi^2}{1 + a_\alpha |\xi|^\alpha}\, \hat{u} 
  = \hat{v_0} \otimes \delta +  \hat{u_0} \otimes \delta'.
\end{equation}
Considered as an ordinary differential equation in $t$ with parameter $\xi$, the latter is solved by
\begin{equation}\label{FTofSol}
\hat{u}(t) = \left( \hat{u_0}(\xi) \cos \left(\frac{\left\vert \xi \right\vert 
t}{\sqrt{1+a_\alpha\left\vert \xi \right\vert^\alpha}} \right)
    + \hat{v_0}(\xi) \, \frac{\sin \left(\frac{\left\vert \xi 
\right\vert t}{\sqrt{1+a_\alpha\left\vert \xi \right\vert^\alpha}} 
\right)}{\frac{\left\vert\xi \right\vert}{\sqrt{1+a_\alpha\left\vert \xi 
\right\vert^\alpha}} } \right) H(t).
\end{equation}
We recall the standard Sobolev spaces $H^s(\R) = \mathcal{F}^{-1} L_s^2(\R)$, where $s \in \R$ and $L_s^2(\R)$ is the set of $L^2$-functions $w$ such that $\xi \mapsto (1 + \xi^2)^{s/2} w(\xi)$ belongs to $L^2$ as well. Let us consider the operator $P$ acting on elements $u\in L^1_{\text{loc}}(\R, H^s(\R)) \cap \S'(\R^2)$ by
\begin{equation}\label{operatorP}
    Pu :=  \partial_t^2 u - L_x^\alpha \partial_x^2 u = 
		\partial_t^2 u - \mathcal{F}_{\xi \rightarrow x}^{-1} 
      (\frac{1}{\sqrt{1 + a_\alpha\left\vert \xi \right\vert^\alpha}})
			\ast_x \partial_x^2 u. 
\end{equation}

\begin{theorem}\label{solvability}
 Let $s \in \R$, $u_0\in H^s(\R)$, and $v_0\in H^{s+1-\alpha/2}(\R)$. Then
\begin{equation}\label{Pu}
Pu = u_0\otimes \delta'  + v_0\otimes \delta  
\end{equation}
has a unique solution $u\in L^1_{\text{\rm loc}}(\R, H^s(\R)) \cap \S'(\R^2)$ with $\supp u \subseteq \{(x,t) \in \R^2 \mid t \geq 0\}$ and 
$u\in C^\infty((0,\infty); \S'(\R))\cap C((0,\infty);H^s(\R))$, 
given by 
\begin{equation}\label{u-fund}
u(t) = u_{0}\ast_{x} E_0(t) 
    + v_0 \ast_x E_1(t),
\end{equation}
where
\begin{eqnarray}
  E_0(t) &:=& \mathcal{F}_{\xi \rightarrow x}^{-1}
	 \left[ \cos \left(\frac{\left\vert \xi \right\vert t}{\sqrt{1+a_\alpha\left\vert \xi \right\vert^\alpha}} \right) H(t) \right], \label{E0}\\
  E_1(t) &:=&  \mathcal{F}_{\xi \rightarrow x}^{-1}
       \left[ \frac{\sin \left(\frac{\left\vert \xi \right\vert t}{\sqrt{1+a_\alpha\left\vert \xi \right\vert^\alpha}} \right)}{\left(\frac{\left\vert 
\xi \right\vert}{\sqrt{1+a_\alpha\left\vert \xi \right\vert^\alpha}} \right)} H(t) 
\right]. \label{E1}
\end{eqnarray}

\end{theorem}
\begin{proof}  Observe that $E_1 \in \S'(\R^2)$ is a fundamental solution of $P$, i.e., $P E_1 = \delta \otimes \delta$. Furthermore, $E_0 = \partial_t E_1$, hence $P E_0 = \delta \otimes \delta'$. Therefore, $u$ given by (\ref{u-fund}) satisfies Equation (\ref{Pu}) in $\S'(\R^2)$. It remains to check the spatial Sobolev regularity.  Since $u_0\in H^s(\R)$ and $v_0\in H^{s+1-\alpha/2}(\R)$ we have  $\widehat{u_0}\in L^2_s(\R)$ and $\widehat{v_0}\in L^2_{s+1-\alpha/2}(\R)$. We observe that
$\widehat{E_0}(\xi,t) = \cos (\frac{\left\vert \xi \right\vert 
t}{\sqrt{1+a_\alpha\left\vert \xi \right\vert^\alpha}}) H(t)$ is bounded and 
$\widehat{E_1}(\xi,t) = \frac{\sqrt{1+a_\alpha\left\vert \xi \right\vert^\alpha}}{\left\vert\xi \right\vert} \sin (\frac{\left\vert \xi \right\vert t}{\sqrt{1+a_\alpha\left\vert \xi \right\vert^\alpha}} ) H(t)   = O(|\xi|^{\frac{\alpha}{2}-1})$, thus as functions of $\xi$ we clearly have 
$\widehat{u_0}\, \widehat{E_0}(t) \in L^2_s(\R)$,  $\widehat{u_1}\, \widehat{E_1}(t) \in L^2_s(\R)$, which implies that $u$ given as in \eqref{u-fund} satisfies $\widehat{u}(t) \in L^2_s(\R) =\mathcal{F}(H^s(\R))$ for every $t$ and $u(t) = 0$, if $t < 0$. With respect to $t$,  the corresponding functions clearly are $L^\infty$ in terms of the $L^2_s(\R)$-norms for $t \in \R$ (hence $L^1_{\text{loc}}$),  weakly smoothly depending on $t >0$ as temperate distributions with respect to $\xi$, and continuous for $t > 0$ with respect to the $L^2_s(\R)$-norms.

To show uniqueness in $L^1_{\text{\rm loc}}(\R, H^s(\R)) \cap\S'(\R^2)$ we have to consider the homogenous equation
$P u = 0$ under the condition $\supp u \subseteq \{(x,t) \in \R^2 \mid t \geq 0\}$.  
Upon Fourier transform with respect to the spatial variable we obtain the ordinary differential equation (with parameter $\xi \in \R$)
$$
 \partial_t^2 \widehat{u} + 
   \frac{\xi^2}{1 + a_\alpha |\xi|^\alpha}\, \widehat{u} 
  = 0,
$$
which has only smooth solutions with respect to time $t$. Since $\widehat{u} \mid_{\{t < 0\}} = 0$, we have, in particular, $\widehat{u}$  that 
satisfies the initial conditions $\widehat{u} \!\mid_{t = -1} = 0$, $\partial_t \widehat{u} \!\mid_{t = -1} = 0$. Therefore, $\widehat{u} = 0$ and thus  $u = 0$. 
\end{proof}

\begin{remark}\label{RemEReg}
Both $E_0$ and $E_1$ are weakly smooth with respect to $t$ when $t \neq 0$, which implied the property $u\in C^\infty((0,\infty); \S'(\R))$ for the solution given in the theorem above. Note that, in addition, we have that 
$t \mapsto E_1(t)$ is continuous $\R \to \S'(\R)$ with $E_1(0) = 0$, whereas $\lim_{t \to 
0+} E_0(t) = \delta \neq 0 = \lim_{t \to 0-} E_0(t)$. However, $E_0$ is weakly measurable with respect to $t \in \R$. 
\end{remark}


\section{Numerical examples}\label{NumEx}

The solution to Eringen's wave equation \eqref{EringenFWE} in the sense of Theorem \ref{solvability} is explicitly given by \eqref{u-fund}. We give a few numerical examples  illustrating the evolution of an initial displacement assumed in the form of a Dirac delta distribution, i.e., $u_0 = \delta \in H^{- \frac{1}{2} - \epsilon}(\R)$ for every $\epsilon > 0$, and with zero initial velocity, i.e., $v_0 = 0$. For the purpose of numerical implementation and 
because the fundamental solution  is  to be convolved with a distribution as well, the initial displacement is taken as the smooth approximation $u_{0, a}$ ($a > 0$), $\lim_{a \to 0} u_{0, a} = \delta$, of the Dirac distribution in the form 
\begin{equation}\label{unulaAprox}
   u_{0,a} (x) :=  \frac{1}{a \sqrt{\pi}} e^{-\frac{x^2}{a^2}}.
\end{equation}
The corresponding approximate solution $u_a$ then is
\begin{align}
u_a(t) = &\, u_{0,a}\ast_{x} E_0(t) =
   \F_{\xi \rightarrow x}^{-1}[\widehat{u_{0,a}}(\xi)\,\widehat{E_0}(\xi,t)] \notag \\
   = &  \, \mathcal{F}_{\xi \rightarrow x}^{-1}
	 \left[\cos \left(\frac{\left\vert \xi \right\vert t}{\sqrt{1+a_\alpha\left\vert \xi \right\vert^\alpha}} \right)  H(t) \, e^{-\frac{a^2 \xi^2 }{4}}  \right], \notag
\end{align}
where we used \eqref{E0} and that $\F_{x \to \xi}(\frac{1}{a \sqrt{\pi}} e^{-\frac{x^2}{a^2}}) = e^{-\frac{a^2 \xi ^2}{4}}$. Upon a straightforward calculation we obtain the following formula which was used in the numerical calculations:
\begin{equation}
 u_a(x,t) = \frac{1}{2 \pi}\int\limits_0^\infty \left[ 
\cos \left(\xi   \left(x+\frac{t}{\sqrt{1+ a_\alpha\xi ^{\alpha }}}\right)\right)+\cos \left(\xi \left(x-\frac{t}{\sqrt{1+ a_\alpha\xi ^{\alpha }}}\right)\right)\right]e^{-\frac{a^2 \xi ^2}{4}}d\xi. \label{u-vNulaJeNula}
\end{equation}

Figures  \ref{fig.1},  \ref{fig.3}, and  \ref{fig.2} show plots of the spatial profiles of the solution $u_a(t)$ at different instances of time as given by \eqref{u-vNulaJeNula} for the regularization parameter value $a=0.1$ and for values of $\alpha$ being $1.5$, $2.5$, and $2$, respectively.
 
\pagebreak

In Figure \ref{fig.1} we observe an interesting oscillatory character of the solution for  $\alpha = 1.5$. A region of strong oscillation seems clearly distinguishable, outside of which the solution vanishes  in fact (though not shown in all figures).  As time progresses, this oscillation region moves along the $x$-axis and also stretches. Moreover, the largest amplitudes  in the oscillations decrease smoothly with time. Overall, such behavior might be interpreted as a sort of propagation of initial disturbances. 
\begin{figure}[htbp]
 \begin{center}
 \begin{minipage}{73mm}
 \includegraphics[scale=0.85]{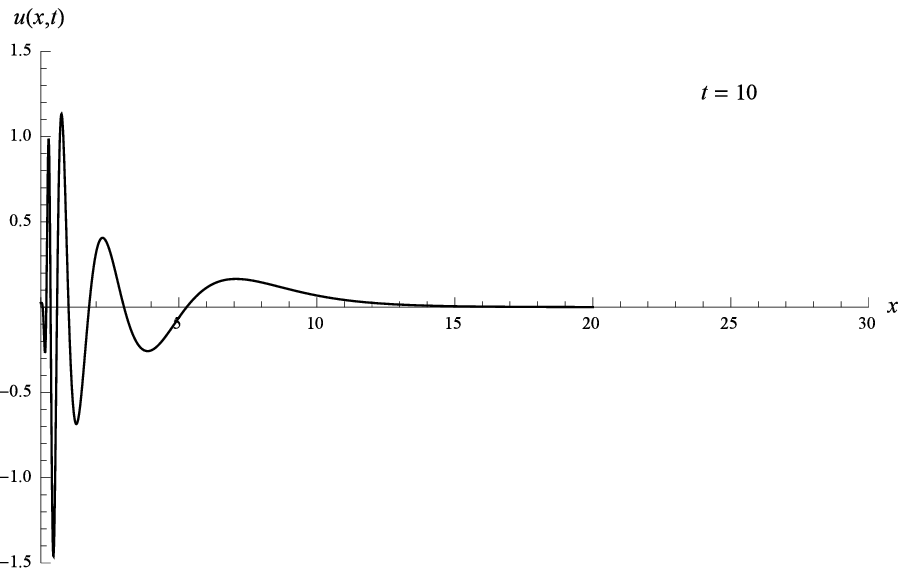}
 \end{minipage}
 \hfil
 \begin{minipage}{73mm}
 \includegraphics[scale=0.85]{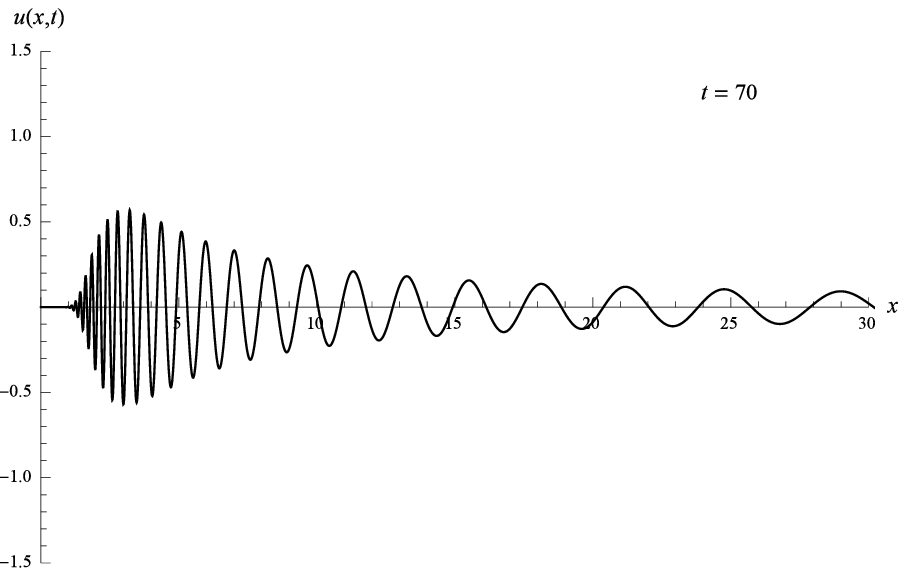}
 \end{minipage}
 \vfil
 \begin{minipage}{73mm}
 \includegraphics[scale=0.85]{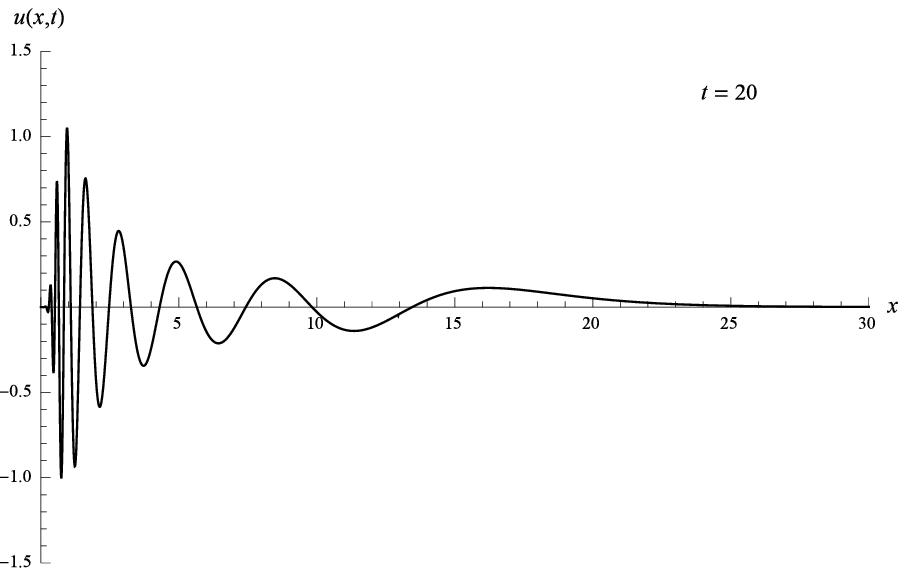}
 \end{minipage}
 \hfil
 \begin{minipage}{73mm}
 \includegraphics[scale=0.85]{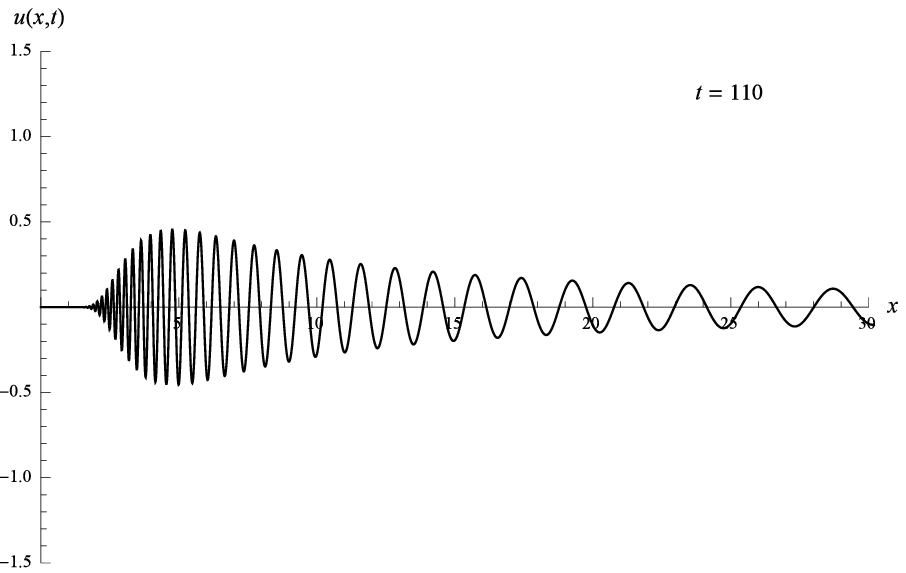}
 \end{minipage}
 \vfil
 \begin{minipage}{73mm}
 \includegraphics[scale=0.85]{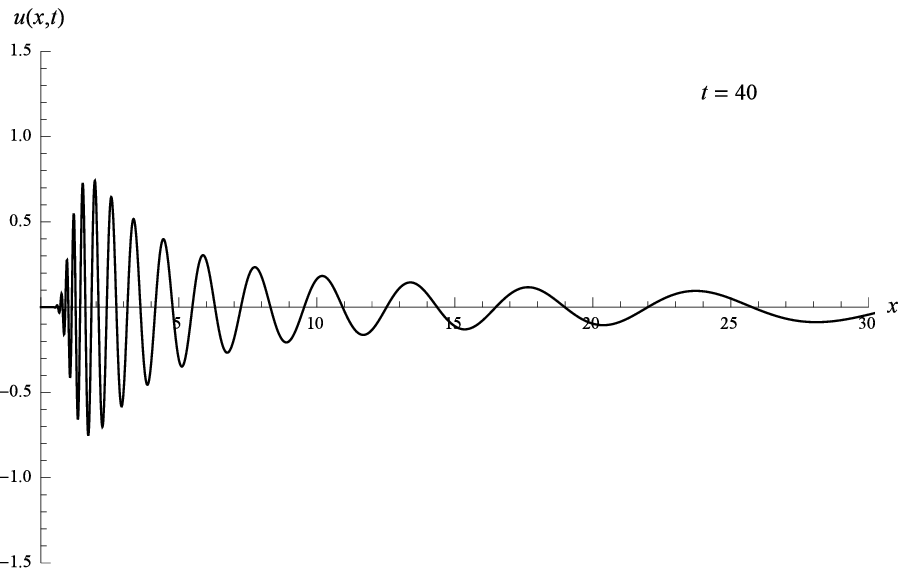}
 \end{minipage}
 \hfil
 \begin{minipage}{73mm}
 \includegraphics[scale=0.85]{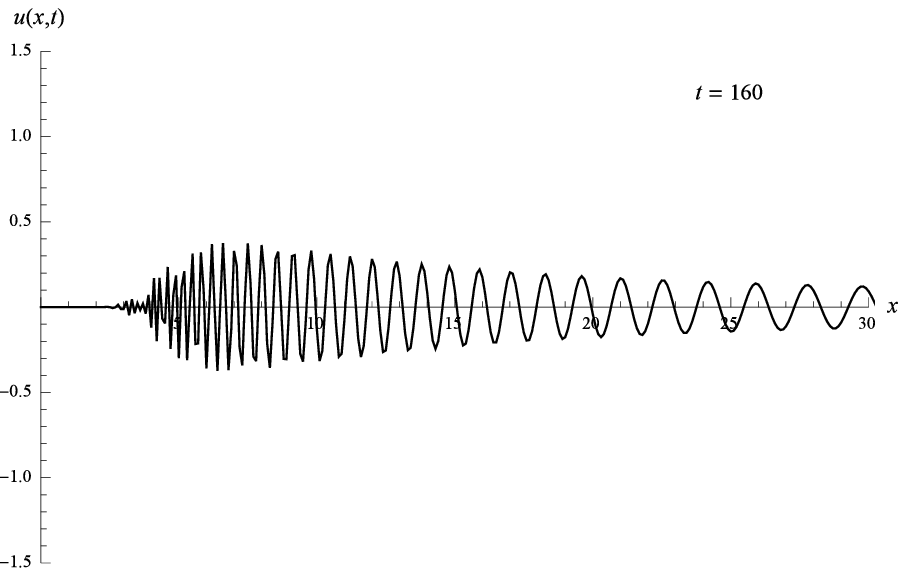}
 \end{minipage}
 \end{center}
 \caption{Spatial profiles of the solution with $\alpha=1.5$ for different instances of time.}
 \label{fig.1}
 \end{figure}
 
\pagebreak 

For $\alpha = 2.5$ the character of the solution shown in Figure \ref{fig.3} is similar to that in the previous case, but the oscillatory character of the solution is less prominent. The region with distinguished oscillations is more narrow, but still stretches as time progresses, and the decrease of the peaks seems less smooth. Moreover, the left boundary of the oscillation region does not seem to move within the considered time interval.  
 
\begin{figure}[htbp]
 \begin{center}
 \begin{minipage}{73mm}
 \includegraphics[scale=0.85]{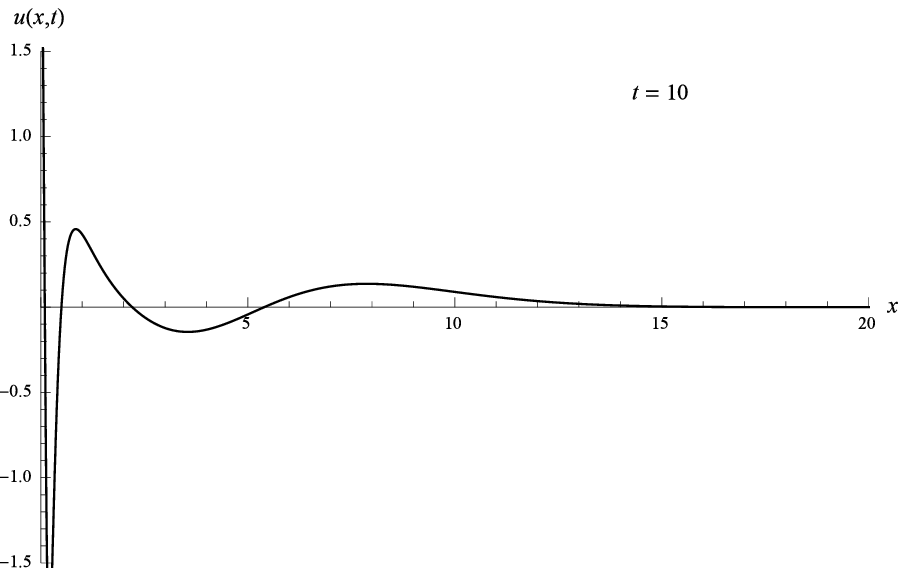}
 \end{minipage}
 \hfil
 \begin{minipage}{73mm}
 \includegraphics[scale=0.85]{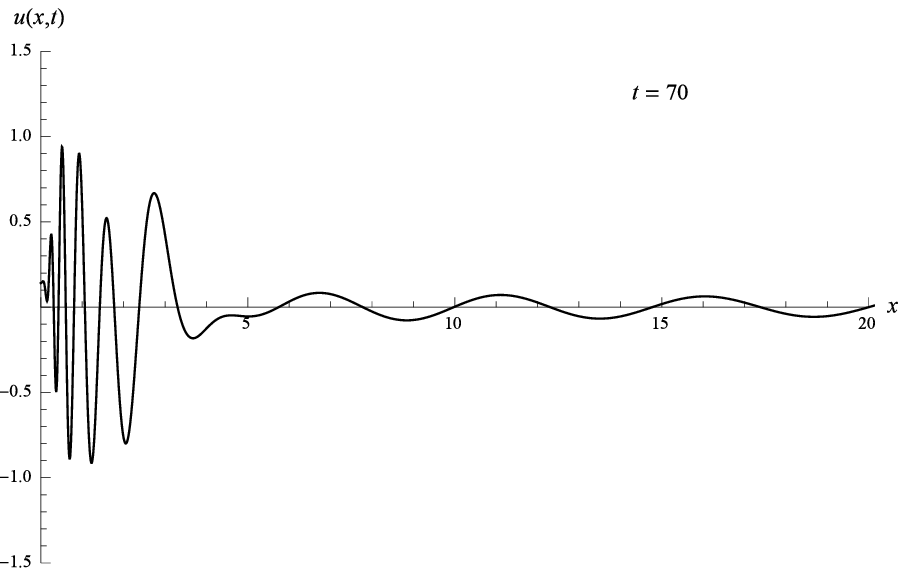}
 \end{minipage}
 \vfil
 \begin{minipage}{73mm}
 \includegraphics[scale=0.85]{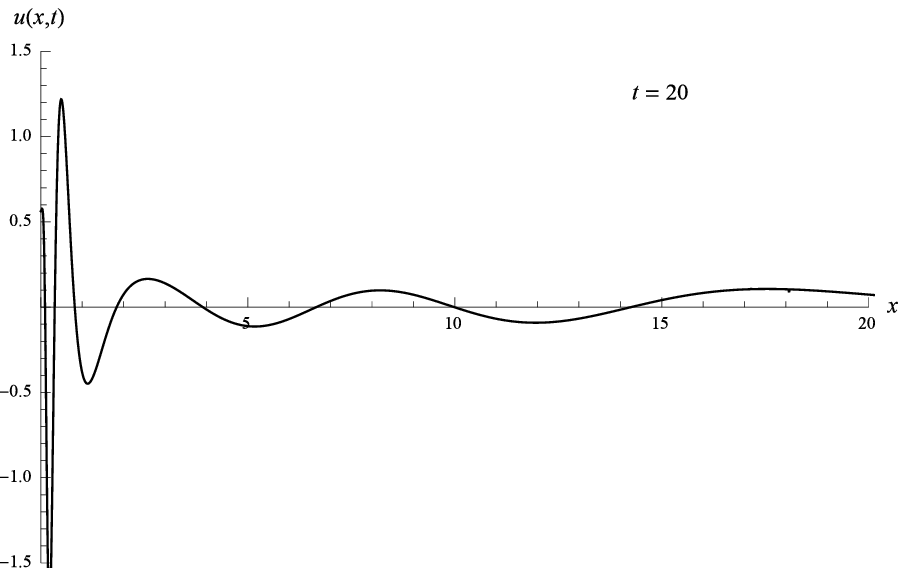}
 \end{minipage}
 \hfil
 \begin{minipage}{73mm}
 \includegraphics[scale=0.85]{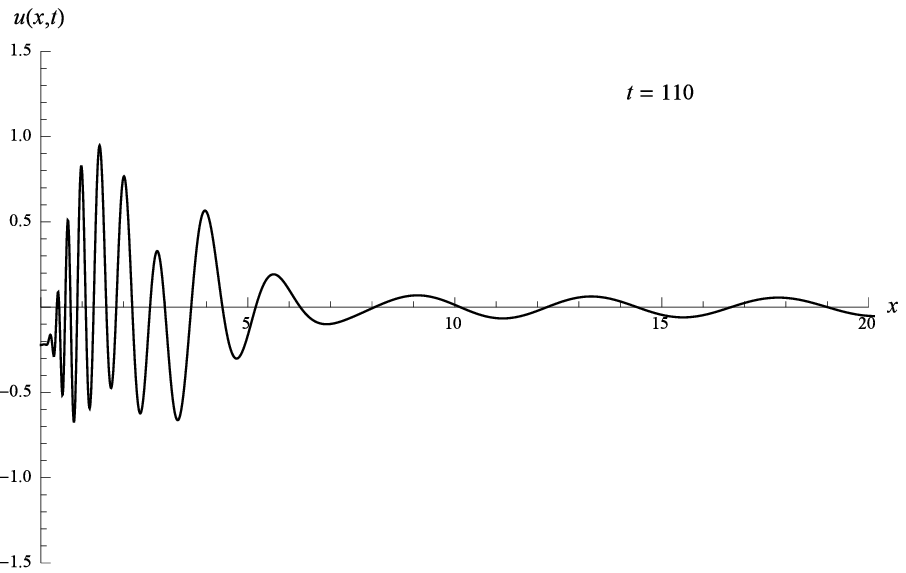}
 \end{minipage}
 \vfil
 \begin{minipage}{73mm}
 \includegraphics[scale=0.85]{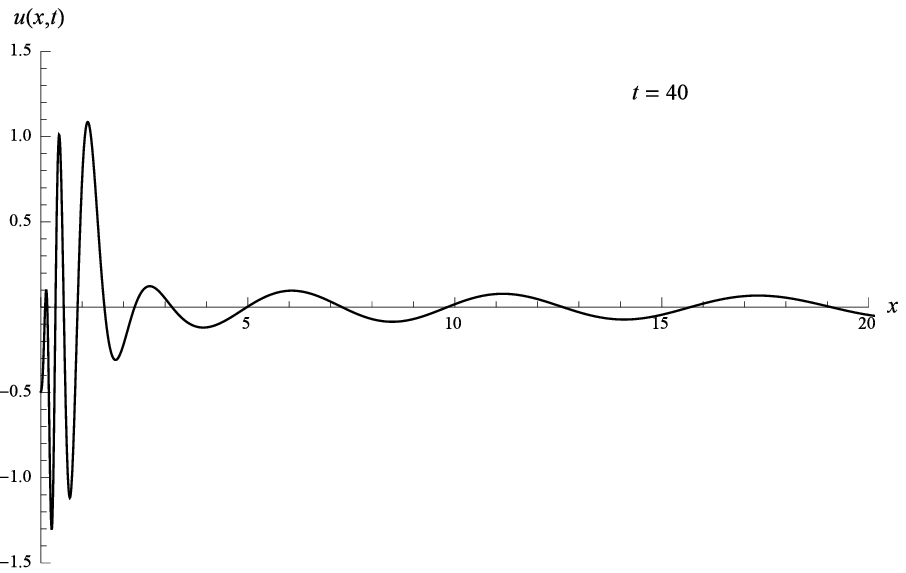}
 \end{minipage}
 \hfil
 \begin{minipage}{73mm}
 \includegraphics[scale=0.85]{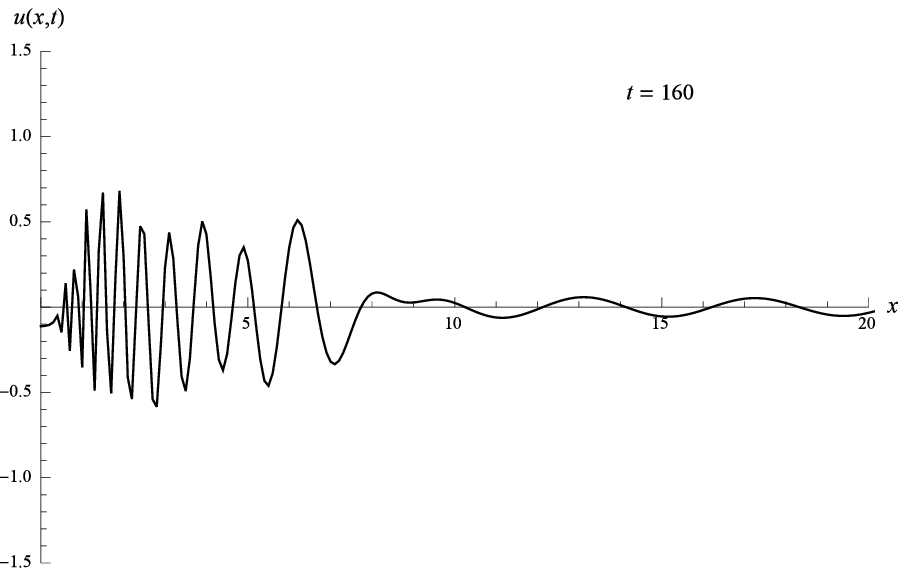}
 \end{minipage}
 \end{center}
 \caption{Spatial profiles of the solution with $\alpha=2.5$ for different instances of time.} 
 \label{fig.3}
 \end{figure}
 \pagebreak
 
 Figure \ref{fig.2} represents the solution to the classical Eringen wave equation, i.e., $\alpha = 2$, and shows an oscillatory character of the solution as well, but there is no particular region of oscillation concentration moving in time. In course of time, the solution stretches and the oscillations are more frequent as in previous cases.  
 
 \begin{figure}[htbp]
 \begin{center}
 \begin{minipage}{73mm}
 \includegraphics[scale=0.85]{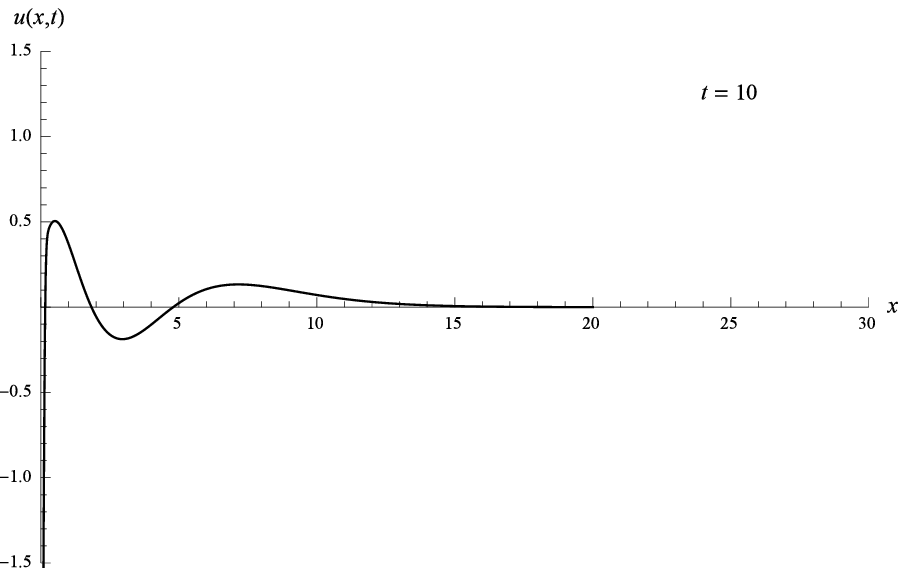}
 \end{minipage}
 \hfil
 \begin{minipage}{73mm}
 \includegraphics[scale=0.85]{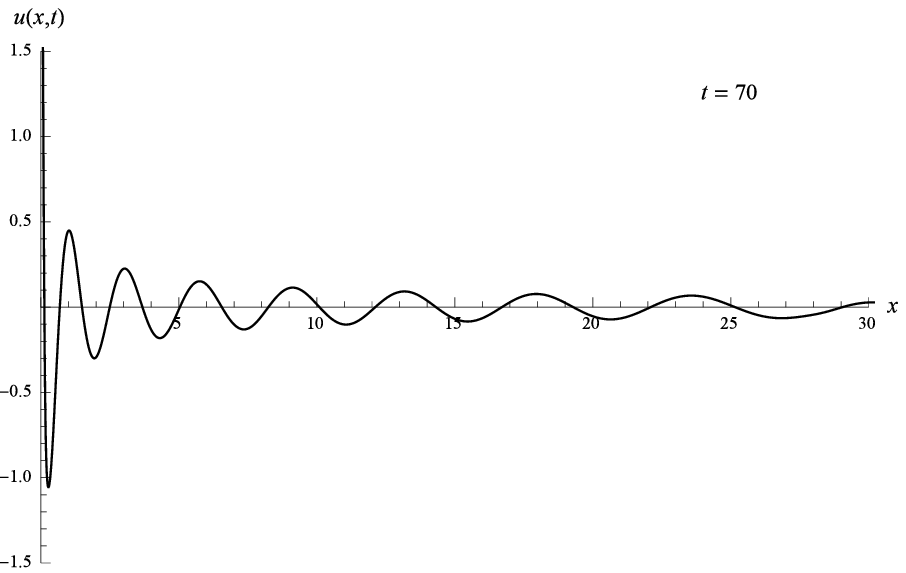}
 \end{minipage}
 \vfil
 \begin{minipage}{73mm}
 \includegraphics[scale=0.85]{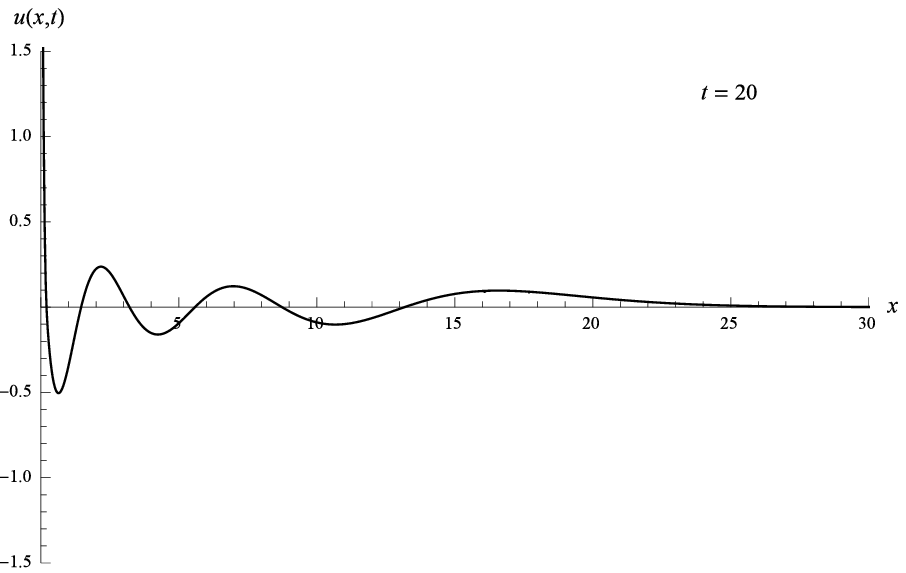}
 \end{minipage}
 \hfil
 \begin{minipage}{73mm}
 \includegraphics[scale=0.85]{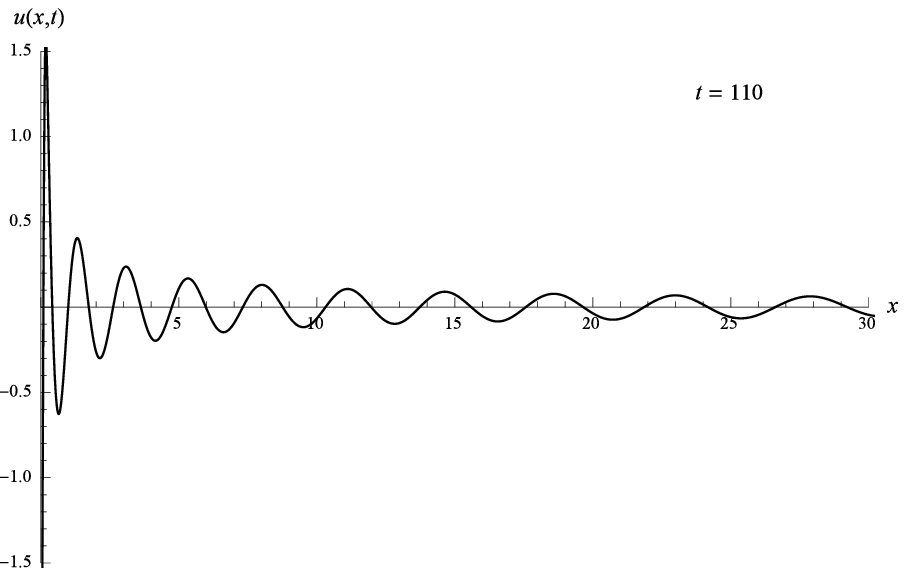}
 \end{minipage}
 \vfil
 \begin{minipage}{73mm}
 \includegraphics[scale=0.85]{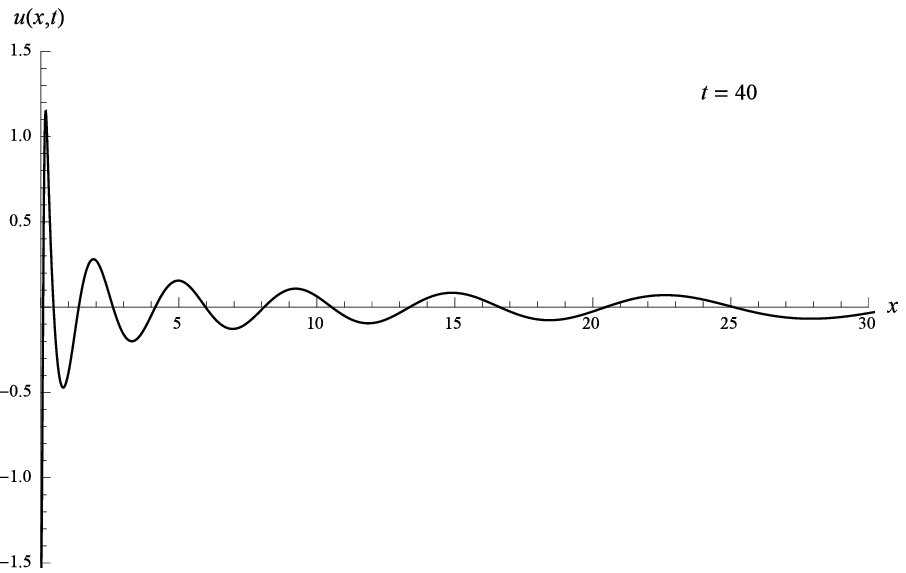}
 \end{minipage}
 \hfil
 \begin{minipage}{73mm}
 \includegraphics[scale=0.85]{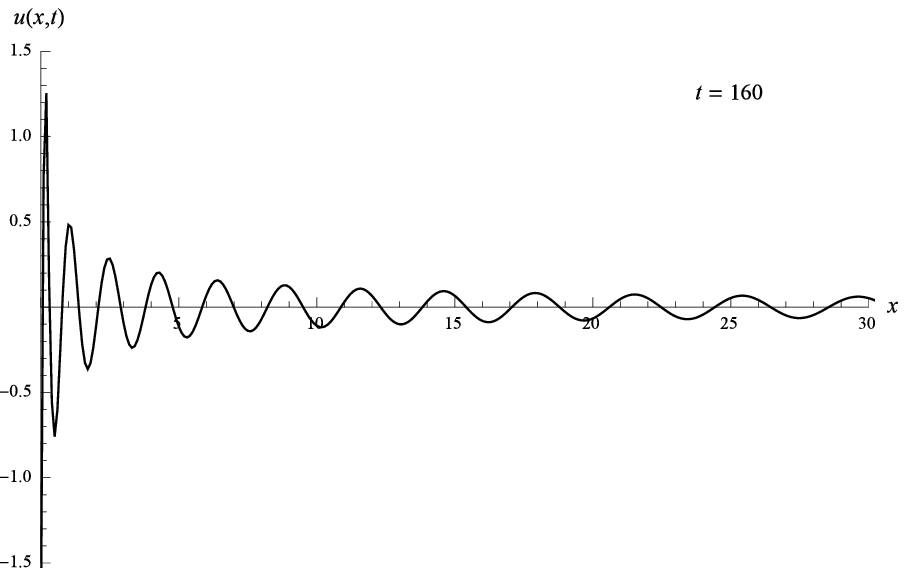}
 \end{minipage}
 \end{center}
 \caption{Spatial profiles of the solution with $\alpha=2$ for different instances of time.}
 \label{fig.2}
 \end{figure}

A first question might be whether the nature of the oscillations in the plots of the numerical approximations stems from singularities in the distributional solution with $u_0 = \delta$. An attempt to address a few aspects of this question microlocally in the case $\alpha \geq 2$ is the content of the following section. A refined analysis of any relatively smooth but oscillatory behavior of the solution and possible propagation of such would seems to require a very different analytical approach and is not subject of the present paper.


\section{The wave front set of the solution in case $\mathbf{\alpha \geq 2}$}\label{MLp}
The symbol of the operator $P$ (in the sense of a Fourier multiplier) is
\begin{equation}
   p(\xi,\tau) = - \tau^2 + \frac{\xi^2}{1 + a_\alpha |\xi|^\alpha} = 
	- \tau^2 + h_\alpha^2(\xi), 
\end{equation}
where
\begin{equation}
 h_\alpha(\xi)= \frac{\left\vert \xi 
\right\vert }{\sqrt{1+a_\alpha\left\vert \xi \right\vert^\alpha}}
\end{equation}
 is a symbol of order $\displaystyle{1 - \frac{\alpha}{2}}$ and of H\"{o}rmander type $(\rho, \delta) = (1,0)$ for large $|\xi|$ (cf.\ \cite[Definition 7.8.1]{hermander}). We note in passing that inclusion relations  as in the lemma and the proposition below can also be shown in case of an arbitrary symbol  in the same class as $h_\alpha$, thus for a slightly wider class of operators than just the Eringen-type model---but the proofs will still require $\alpha \geq 2$.

\begin{lemma}\label{fund-res}  Let $\alpha \geq 2$ and  $E_j^+$ ($j=0,1$) denote the 
restriction of $E_j$  to the open half-space $\{(x,t)\in \R^2 \mid t > 0\}$. Then the wave front sets satisfy
$$
   \WF(E_j^+) \subseteq 
      \{ (0,t;\xi,0) \mid t > 0, \xi \neq 0\}=:W_0.
$$
\end{lemma}

\begin{proof} We offer two variants of this proof. 

\paragraph{Variant (A):} This is an adaptation of the proof of Claim 1 from the proof of \cite[Lemma 2.2]{HOZ16}. For convenience of the reader, we provide the details.

Let $t > 0$, put $e_0^+(\xi,t) := \cos\left(\frac{\left\vert \xi \right\vert 
t}{\sqrt{1+a_\alpha\left\vert \xi \right\vert^\alpha}} \right)$ and 
$e_1^+(\xi,t) := \sin \left(\frac{\left\vert \xi \right\vert 
t}{\sqrt{1+a_\alpha\left\vert \xi \right\vert^\alpha}} \right)/
\left(\frac{\left\vert \xi \right\vert}{\sqrt{1+a_\alpha\left\vert \xi 
\right\vert^\alpha}} \right)$, and choose $\rho \in C^{\infty}(\R)$ such that $\rho(\xi) = 0$ for 
$|\xi| \leq 1/2$ and $\rho (\xi) =1$ for $|\xi| \geq 1$.  Then at fixed arbitrary $t > 0$ and $j \in 
\{0,1\}$ we have
$$
E_j^+(t) = \mathcal{F}_{\xi \to x}^{-1} \left(e_j^+(\xi,t)\right) = \mathcal{F}^{-1}_{\xi \to x} \left(e_j^+(\xi,t) \rho(\xi)\right) + \mathcal{F}_{\xi \to x}^{-1} \left(e_j^+(\xi,t) (1 - \rho(\xi))\right) =: F_{j,1}(t) + F_{j,2}(t).
$$
We observe that $(\xi,t) \mapsto e_j^+(\xi,t) (1 - \rho(\xi))$ is continuous, has compact $\xi$-support, and is smooth with respect to $t$, more precisely, $e_j^+ (1 - \rho) \in C^\infty(]0,\infty[, C_{\text{c}}(\R))$, hence by linearity, commutativity with $\frac{d}{dt}$, and continuity of the inverse Fourier transform with respect to $\xi$, we have $F_{j,2} \in C^\infty \left( ]0,\infty[, \mathcal{F}^{-1}\left(C_{\text{c}}(\R)\right) \right) \subseteq C^\infty \left( ]0,\infty[, C^\infty(\R) \right) \subseteq C^\infty(]0,\infty[ \, \times \R)$, thus, $\WF(E_j^+) = \WF(F_{j,1})$.

Note that $a_j(x,t,\xi) := e_j^+(\xi,t) \rho(\xi)$ define functions in $C^\infty(\R^2\times \R)$ 
($j=0,1$) and that $a_1(x,t,\xi) = \int_0^t a_0(x,s,\xi) \, ds$. Furthermore, $a_0$ and $a_1$ are symbols of class $S^0(\R^2\times \R)$  by \cite[Proposition 1.1.8]{hermander-paper}. 

We observe that $F_{j,1}$ ($j=0,1$), being an inverse Fourier transform, can be written as oscillatory integral (in the sense of \cite[Theorem 7.8.2]{hermander}) with symbol $a_j(x,t,\xi)/(2\pi)$ and phase function $\phi(x,t,\xi) = x \, \xi$ in both cases. Thus, according to \cite[Theorem 8.1.9]{hermander}, the only contributions to the wave front sets can stem from points with stationary phase, i.e.,
\begin{eqnarray*}
 \WF\left( E_j^+ \right) &\subseteq &\left\{ \left( x,t; \partial_x \phi(x,t,\xi), \partial_t \phi(x,t,\xi) \right) \mid t > 0, \exists \, \xi \neq 0 \colon \partial
_{\xi }\phi \left( x,t,\xi \right) =0\right\}  \\
&=&\left\{ \left( 0,t,\xi ,0\right) \mid t > 0, \xi \neq 0\right\} = W_{0}.
\end{eqnarray*}

\paragraph{Variant (B)}  As a preparation we need a technical sublemma, which is very similar to Lemma 2.5  in \cite{HOZ16} (hence we do not repeat its proof). 

\textbf{Sublemma:} Let $\Gamma \subseteq \R^2$ (representing the $(\xi,\tau)$-plane) be the union of a closed disc around $(0,0)$ and a closed narrow cone containg the $\tau$-axis and being symmetric with respect to both axes. Let $\Gamma'$ be a closed set of the same shape as $\Gamma$, but with slightly larger disc and opening angle of the cone. Let $\tilde{b}\in S^{0}\left( \mathbb{R}^{2}\times \mathbb{R}^{2}\right)$ such that $\tilde{b}(x,t,\xi,\tau)$ is real, constant with respect to $(x,t)$, homogenous of degree $0$ with respect to $(\xi,\tau)$  away from the disc contained in $\Gamma'$, and such that $\tilde{b}(x,t,\xi,\tau) = 0$, if $(\xi,\tau) \in \Gamma$, $\tilde{b}(x,t,\xi,\tau) = 1$, if $(\xi,\tau) \not\in \Gamma'$. Then $p \, \tilde{b}$ is a symbol belonging to the class $S^{2}\left( \mathbb{R}^{2}\times
\mathbb{R}^{2}\right)$.

 We put $q := p \cdot \tilde{b}$ with $\tilde{b}$ as in the sublemma and denote by $Q$ be the pseudodifferential operator with symbol $q$. Then we have $Q E_j^+ = 0$ and therefore
$\WF(E_j^+) \subseteq \Char(Q)$ holds. Noting that $\Char(Q) \subseteq \R^2 \times (\Gamma'  \cup \{ (\xi,0) \mid \xi \neq 0 \} )$ and we may choose $\Gamma'$ arbitrarily narrow around the $\tau$-axis (for $|\tau| \geq 1$), we may conclude that 
\begin{equation}\label{CharRegE+}
 \WF(E_j^+) \subseteq \R^2 \times (\{(0,\tau) \mid \tau \neq 0\} 
	    \cup \{ (\xi,0) \mid \xi \neq 0 \}).
\end{equation}
As noted in Remark \ref{RemEReg}, the map $t \mapsto E_j^+(t)$ is weakly smooth $]0,\infty[ \to \S'(\R)$, hence by \cite[(23.65.5)]{diodone} no cotangent point $(x,t;0,\tau)$ can be contained in $\WF(E_j^+)$. Therefore, \eqref{CharRegE+} can be strenghtened to the inclusion
\begin{equation}\label{WFE+1}
 \WF(E_j^+) \subseteq \R^2 \times \{ (\xi,0) \mid \xi \neq 0 \}.
\end{equation}
As noted in Variant (A), $h_\alpha$ is a symbol in $S^0(\R^2\times\R^2)$ for $\xi\neq 0$ (due to $\alpha \geq 2$ in our hypothesis), thus $e_j^+(\xi,t)$ are in $S^0(\R^2\times R^2)$ (apart from a cut-off near $\xi = 0$) and therefore $\singsupp(E_j^+) \subseteq \{(0,t) \mid t > 0\}$ by the usual stationary phase argument. In conclusion, we obtain
$\WF(E_j^+) \subseteq W_0$.
\end{proof}

Based on Lemma \ref{fund-res} we will proceed to investigate
the influence of the singularities in the initial data $u_0$ and
$v_0$ on the wave front set of the solution $u$ to (\ref{Pu}). As a first result, we obtain a wave front set inclusion relation in the following proposition. The proof is an exact repetition of the first part of the proof of \cite[Theorem 2.4]{HOZ16} (up to Equation (13) there) and is based on \cite[Theorems 8.2.4 and 8.2.13]{hermander}. Hence we do not include the details here.

\begin{proposition} \label{June} Let $\alpha \geq 2$, $u_{0}, v_0\in \mathcal{E}^{\prime
}\left(\mathbb{R}\right)$,s and denote by $u^+$ the restriction of the solution $u$ according to Theorem \ref{solvability} to the half-space of future time $\R \times \, ]0,\infty[$, then 
\begin{equation*}
\WF\left( u^+ \right) \subseteq \left\{ \left( x,t ; \xi ,0\right) \mid t > 0,  \left(
x,\xi \right) \in WF\left( u_{0}\right) \text{ or }
  \left( x,\xi \right) \in WF\left( v_{0}\right) \right\}.
\end{equation*}
\end{proposition}

\begin{remark}\label{importantremark} (i) The result on the wave front set of $u^+$ in the above proposition implies, in particular, smoothness of $u^+$ considered as a map from time into distributions on space  (cf.\   \cite[(23.65.5)]{diodone}), i.e., $u^+ \in C^{\infty }( ]0,\infty[,\mathcal{D}^{\prime }(\mathbb{R}))$; in addition, we have $u^+(t) \in \mathcal{S}'(\R)$ for every $t > 0$.

(ii) We observe that $\partial_t E_1^+ = E_0^+$ implies
$$
   \WF (E_0^+) \subseteq \WF(E_1^+) \subseteq
      \WF(E_0^+) \cup \{ (x,t ; \xi,0) \mid t > 0, \xi \neq 0 \}, 
$$
where the right-most set corresponds to the characteristic set of $\partial_t$ when considered as partial differential operator on $\R \times \,]0,\infty[$.
\end{remark}

\begin{theorem} Let the hypothesis of Proposition \ref{June} hold. If $v_0$ is smooth then 
$$
  \WF\left( u^+ \right) = \left\{ \left( x,t ; \xi ,0\right) \mid t > 0,  \left(x,\xi \right) \in \WF\left( u_{0}\right) \right\},
$$
and similarly  if $u_0$ is smooth then
$$\WF\left( u^+ \right) = \left\{ \left( x,t ; \xi ,0\right) \mid t > 0, \left(x,\xi \right) \in \WF\left( v_{0}\right) \right\}.$$
\end{theorem}

\begin{proof} The overall strategy of the proof is the same as in that of \cite[Theorem 2.4]{HOZ16} and a part of the latter can, in fact, be taken over almost literally. We will indicate the corresponding step in the current proof and refer there for details to the earlier constructions in \cite{HOZ16}. 

Let $\tilde{E}(t)  := \mathcal{F}_{\xi \rightarrow x}^{-1} \left[ \exp\left(i t h_\alpha(\xi)\right) \right]$ and, for any $t \in \R$, put $\tilde{u}(t) := \tilde{E}(t) \ast u_0$. We have
$ D_t \widehat{\tilde{E}(t)} =
  \frac{1}{i}\partial _{t} \widehat{\tilde{E} (t)} =  
	h_\alpha(\xi)\, \mathrm{e}^{\mathrm{i}t h_\alpha(\xi)} = h_\alpha(\xi) \widehat{\tilde{E}(t) }$,
which implies
\begin{equation*}
 Y \tilde{E} :=  -D_{t}\tilde{E} +A_{x}^{\alpha}\tilde{E} =0,
 \quad \tilde{E}\left( 0\right) =\delta,
\end{equation*}
where $A_{x}^{\alpha}w = \mathcal{F}_{\xi \rightarrow
x}^{-1}\left[ h_\alpha(\xi) \right] \ast_x w$ with $w \in H^{-\infty}(\R)$. Moreover, $\tilde{u}$ solves the initial
value problem
\begin{equation}
 Y \tilde{u} =
 \left( -D_{t}+A_{x}^{\alpha}\right) \tilde{u}
=\left( -D_{t} \tilde{E} +A_{x}^{\alpha} \tilde{E} \right) \ast_x u_0 = 0,
  \quad \tilde{u}(0) =u_{0}. \tag{$*$}
\end{equation}
The function $h_\alpha^2(\xi)$ matches the symbol of $- E_x^\alpha \, \partial_x^2$, hence we obtain a commutative factorization of $P$ by $\left( D_{t}+A_{x}^{\alpha}\right) \left( -D_{t}+A_{x}^{\alpha}\right) v=-D_{t}^{2}v+A_{x}^{\alpha }A_{x}^{\alpha }v=\partial
_{t}^{2}v+\mathcal{F}_{\xi \rightarrow x}^{-1} (h_\alpha^2(\xi)) \ast_x v = Pv$, i.e., $P = \bar{Y} \cdot Y = Y \cdot \bar{Y}$, where we have put $\bar{Y} := D_{t}+A_{x}^{\alpha}$.

From here onward and using the same notation (except having now $P$ in place of $Z$), we may follow the detailed construction in the proof of Theorem 2.4 in \cite{HOZ16}, which is carried out between the equation also labeled ($*$) there and Equation (16) of \cite{HOZ16}. From that construction we only need to note one result following essentially from propagation of singularities for $\tilde{u}$ , namely
\begin{equation}
  \WF( \tilde{u}) = \{ (x,t;\xi ,0) \mid ( x,\xi ) \in \WF( u_{0}), t \in \R \} =: W_1,
\label{prop-sing}
\end{equation}
and, in addition, that in the open region with $t > 0$ we have
\begin{equation} \label{YA}
  \bar{Y} u^+ =\left( D_{t} E_0^+ +A_{x}^{\alpha }E_0^+ \right) \ast_x u_{0} =
 A_{x}^{\alpha}\tilde{E} \ast_x u_{0} =
  A_{x}^{\alpha}\tilde{u}.
\end{equation}

Let $\tilde{u}^+$ be the restriction of $\tilde{u}$ to the open half-plane $t > 0$. In the remainder the proof, which follows again the strategy for proving Theorem 2.4 in \cite{HOZ16}, but differs more substantially in the details, we will deduce from \eqref{YA} the relation
\begin{equation}
\WF (\tilde{u}^+) \subseteq \WF (u^+). \label{u-tilda-u}
\end{equation}
Once this is established, we obtain by \eqref{prop-sing} and Proposition \ref{June} that
$$
   W_1 \!\!\mid_{\{t > 0\}} = \WF(\tilde{u}^+) \subseteq \WF(u^+) 
	   \subseteq W_1 \!\!\mid_{\{t > 0\}},
$$
hence the theorem will be proved.

It remains to prove \eqref{u-tilda-u}. To begin with, we note that
\begin{equation} \tag{A}
   \WF(A_x^\alpha v) \subseteq \WF(v),
\end{equation} 
for any $v \in C(]0,\infty[, H^s(\R))$ implies 
\begin{equation} \tag{I}
   \WF(\bar{Y} u^+) \subseteq \WF(u^+),
\end{equation}
since $\bar{Y} = D_t + A_x^\sigma$ and $D_t$ clearly is a microlocal, i.e, $\WF(D_t v) \subseteq \WF(v)$ for any $v \in \mathcal{D}'(\R^2)$. As a second preparation, we aim for the result 
\begin{equation} \tag{II}
   \WF(\tilde{u}^+) \subseteq \WF(A_x^\alpha \tilde{u}^+),
\end{equation}
because this then yields
$$
   \WF(\tilde{u}^+) \subseteq \WF(A_x^\alpha \tilde{u}^+) = 
	   \WF(\bar{Y} u^+) \subseteq \WF(u^+),
$$
hence \eqref{u-tilda-u}. 

The proof of (II) will be based on the relation 
\begin{equation} \tag{B}
   \WF(B_x^\alpha v) \subseteq \WF(v)
\end{equation} 
for any $v \in C(]0,\infty[, H^s(\R))$, where 
$$
   B_x^\alpha v := \mathcal{F}_{\xi \to x}^{-1} \left( 
	     g_\alpha(\xi) \widehat{v(t)}(\xi) 	\right), 
			\quad g_\alpha(\xi) := \frac{i \xi}{h_\alpha(\xi)} = 
			i \,\text{sgn}(\xi) \sqrt{1 +a_\alpha |\xi|^\alpha}. 
$$
Indeed, first note that $A_x^\alpha B_x^\alpha v = B_x^\alpha A_x^\alpha v = \partial_x v$ and (B) gives
$$
	\WF(\partial_x \tilde{u}^+) = \WF(B_x^\alpha A_x^\alpha \tilde{u}^+) 
	   \subseteq \WF(A_x^\alpha \tilde{u}^+).
$$
Second, the map $t \mapsto \tilde{u}(t)$ clearly is weakly smooth, hence $\WF(\tilde{u}) \cap (\R^2 \times \{ (0,\tau) \mid \tau \neq 0 \}) = \emptyset$ (again by \cite[(23.65.5)]{diodone}) and $\WF(\tilde{u}) \subseteq \WF(\partial_x \tilde{u}) \cup (\R^2 \times \{ (0,\tau) \mid \tau \neq 0 \})$ by non-characteristic regularity, thus
$$
    \WF(\tilde{u}) = \WF(\partial_x \tilde{u}),
$$
which finally yields (II).

After all these preparations, our final tasks are to prove (A) and (B).

From the asymptotic properties of the symbols defining the operators $A_x^\alpha$ and $B_x^\alpha$, we have the following mapping properties with respect to spatial Sobolev regularity: $A_x^\alpha$ maps  $C(]0,\infty[,H^s(\R)) \to 
	    C(]0,\infty[,H^{s - 1 + \frac{\alpha}{2}}(\R))$ and $B_x^\alpha$ maps 
			$C(]0,\infty[,H^r(\R)) \to  C(]0,\infty[,H^{r - \frac{\alpha}{2}}(\R))$.

Let $H_\alpha := \mathcal{F}^{-1} h_\alpha$ and $G_\alpha := \mathcal{F}^{-1} g_\alpha$, then $A_x^\alpha v = (H_\alpha \otimes \delta) \ast v$ and $B_x^\alpha v = (G_\alpha \otimes \delta) \ast v$. Therefore, (A) and (B) is established, if we show that $\singsupp(H_\alpha) = \singsupp(G_\alpha) = \{ 0\}$, because then the convolution operators $A_x^\alpha$ and $B_x^\alpha$ do not enlarge the wave front sets, which follows along the line of arguments using the distribution kernels of the convolutions  exactly as in the proof of Equation (13) in \cite{HOZ16} (based on \cite[Theorem 8.2.13]{hermander}).

Observe that $G_\alpha = \frac{d}{dx} H_\alpha$ and $\frac{d}{dx}$ is an elliptic differential operator on $\R$, thus $\singsupp(H_\alpha) = \singsupp(G_\alpha)$. 

Let $\chi$ be smooth with compact support on $\R$ such that $\chi(\xi) = 1$ for $\xi$ in a neighborhood of $0$. Put $r_\alpha := (1 - \chi) g_\alpha$ and note that
$$
   G_\alpha = \F^{-1} g_\alpha = \F^{-1}(\chi g_\alpha) + \F^{-1} r_\alpha,
$$ 
where the first term on the right-hand side is smooth, since $\chi g_\alpha$ has compact support. Therefore
$\singsupp(G_\alpha) = \singsupp(\F^{-1} r_\alpha)$.

The function $r_\alpha$ is smooth and its derivatives are easily seen to yield symbol estimates (calculations are only required for $\xi \not\in \supp(\chi)$ and recall that $\alpha \geq 2$), thus $r_\alpha \in S^{\frac{\alpha}{2}}(\R\times \R)$. Moreover, for large $|\xi|$ we obtain   
$$
  |r_\alpha(\xi)| \geq \sqrt{\frac{|a_\alpha|}{2}} |\xi|^{\alpha/2} \geq c_0 (1 + |\xi|)^{\alpha/2},
$$
which shows that $r_\alpha$ is an elliptic symbol. Thus, we conclude from $\F^{-1} r_\alpha = r_\alpha(D) \delta$ that
$$
  \singsupp (\F^{-1} r_\alpha) = \singsupp(r_\alpha(D) \delta) = \singsupp(\delta) = \{0\}.
$$

\end{proof}

\section*{Acknowledgement}

This research is supported by project P25326 of the Austrian
Science Fund and by projects $174005$, $174024$ of the Serbian
Ministry of Education and Science, as well as by project
$114-451-2098$ of the Secretariat for Science of Vojvodina.


\end{document}